\documentclass{amsart}
\usepackage{latexsym,amsxtra,amscd,ifthen}
\usepackage{amsfonts}
\usepackage{verbatim}
\usepackage{amsmath}
\usepackage{amssymb}
\numberwithin{equation}{section}
\theoremstyle{plain}
\newtheorem{theorem}{Theorem}[section]
\newtheorem{lemma}[theorem]{Lemma}
\newtheorem{proposition}[theorem]{Proposition}

\newcommand{\inth}{\textstyle \int}
\theoremstyle{definition}
\newtheorem{definition}[theorem]{Definition}
\newtheorem{example}[theorem]{Example}
\newtheorem{remark}[theorem]{Remark}
\newtheorem{question}[theorem]{Question}

\newcommand{\tails}{\ensuremath{\operatorname{tails}}}
\makeatletter              
\let\c@equation\c@theorem  
\makeatother
\DeclareMathOperator{\gldim}{gldim} 
\DeclareMathOperator{\Ext}{Ext}  
\DeclareMathOperator{\grmod}{grmod}
\DeclareMathOperator{\tors}{tors}
\DeclareMathOperator{\injdim}{injdim}
\DeclareMathOperator{\GKdim}{GKdim}
\DeclareMathOperator{\End}{End}
 \DeclareMathOperator{\Hom}{Hom}
\def\Pty{{\sf p}}
\begin{document}
\title[Auslander's Theorem for down-up algebras]
{Auslander's Theorem for group coactions on \\
noetherian graded down-up algebras}

\author{J. Chen, E. Kirkman and J.J. Zhang}

\address{Chen: School of Mathematical Sciences,
Xiamen University, Xiamen 361005, Fujian, China}

\email{chenjianmin@xmu.edu.cn}

\address{Kirkman: Department of Mathematics,
P. O. Box 7388, Wake Forest University, Winston-Salem, NC 27109}

\email{kirkman@wfu.edu}

\address{Zhang: Department of Mathematics, Box 354350,
University of Washington, Seattle, Washington 98195, USA}

\email{zhang@math.washington.edu}

\begin{abstract}
We prove a version of a theorem of Auslander for finite group coactions 
on noetherian graded down-up algebras.
\end{abstract}







\maketitle


\setcounter{section}{-1}

\section{Introduction}
\label{xxsec0}
Maurice Auslander \cite{Au} proved that if $G$ is a finite subgroup of 
${\rm{GL}}_n(\Bbbk)$, containing no pseudo-reflections (e.g. subgroups 
of ${\rm{SL}}_n(\Bbbk)$), acting linearly on the commutative polynomial
ring $A= \Bbbk[x_1, \dots, x_n]$, with fixed subring $A^G$, then the 
natural map from the skew group algebra $A*G$ to $\End_{A^G}(A)$ is an 
isomorphism of graded algebras. This theorem is the main ingredient in 
the McKay correspondence, relating representations of $G$ and $A^G$-modules.
Noncommutative versions of this theorem of Auslander \cite{BHZ1, BHZ2} are 
an important ingredient in establishing a noncommutative McKay correspondence. 
One of the main open questions concerning a noncommutative version of 
Auslander's theorem is the following conjecture that was stated in 
\cite[Conjecture 0.4]{BHZ1} and \cite[Conjecture 0.2]{CKWZ2}, where the 
condition that the homological determinant of the $H$-action is trivial 
generalizes the result for group actions by subgroups of SL$_n(\Bbbk)$:\\

\noindent
{\it 
Let $A$ be a connected graded noetherian Artin-Schelter regular algebra
\cite{AS} and $H$ be a semisimple Hopf algebra acting on $A$ inner-faithfully and 
homogeneously. If the homological determinant of the $H$-action on $A$ 
is trivial, then there is a natural graded algebra isomorphism
$$A\#H \cong \End_{A^H}(A).$$
}

By \cite[Theorem 0.3]{CKWZ2} the above conjecture holds when $A$ has 
global dimension two, which is one of the main results in \cite{CKWZ2}. 
It is natural to search for a proof of this conjecture for global 
dimension three (or higher). The paper \cite{BHZ2} started this program 
by showing that the above conjecture holds for certain finite group 
actions on noetherian graded down-up algebras, which are 
Artin-Schelter regular algebras of global dimension three 
\cite[Theorem 0.6]{BHZ2}. Some interesting partial results concerning 
Auslander's theorem have been proven in \cite{BHZ1, BHZ2, GKMW, HZ, MU}. 
The goal of this paper is to verify the conjecture for finite group 
coactions on Artin-Schelter regular down-up algebras [Theorem \ref{xxthm0.1}]. 
The idea of the proof is to use the pertinency introduced in \cite{BHZ1} 
that has been one major tool for proving the noncommutative Auslander theorem.

Throughout the paper, let $\Bbbk$ be a base field of characteristic 
zero, and all objects are over $\Bbbk$.

Down-up algebras were introduced in 1998 by Benkart-Roby in \cite{BR}, and,
since then, these algebras have been studied extensively. Noetherian 
graded down-up algebras are Artin-Schelter regular algebras of global 
dimension three with two generators by a result of \cite{KMP}. Let 
$\alpha$ and $\beta$ be two scalars in $\Bbbk$. The graded down-up 
algebra, denoted by ${\mathbb D}(\alpha,\beta)$, is generated by two 
elements $d$ and $u$ and subject to two relations
\begin{align}
\label{E0.0.1}\tag{E0.0.1}
d^2 u &= \alpha dud +\beta u d^2,\\
\label{E0.0.2}\tag{E0.0.2}
d u^2 &= \alpha udu +\beta u^2 d.
\end{align}
This algebra is noetherian if and only if $\beta\neq 0$, and in this 
paper we always assume that $\beta\neq 0$. When $\alpha=0$, we use 
${\mathbb D}_\beta$ instead of ${\mathbb D}(0,\beta)$. The groups of  graded
algebra automorphisms of the down-up algebras were computed in \cite{KK}. 
Recently, the invariant theory of graded down-up algebras under finite 
group actions and coactions has been studied in \cite{KKZ2, CKZ, HZ}.

In a general setting, let $H$ be a semisimple Hopf algebra and let 
$K$ be its $\Bbbk$-linear dual. Then $K$ is also a semisimple 
Hopf algebra. It is well-known that a left $H$-action on an algebra 
$A$ is equivalent to a right $K$-coaction on $A$. 

Suppose $H$ is a semisimple Hopf algebra with integral $\inth$, and 
$A$ is an algebra with $\GKdim A<\infty$. Here $\GKdim A$ denotes 
the Gelfand-Kirillov dimension of $A$. If $H$ acts on $A$, by 
\cite[Definition 0.1]{BHZ1}, the {\it pertinency} of the $H$-action 
on $A$ is defined to be
\begin{equation}
\label{E0.0.3}\tag{E0.0.3}
\Pty(A,H)=\GKdim A-\GKdim ((A\# H)/I)
\end{equation}
where $I$ is the 2-sided ideal of $A\# H$ generated by $1\# \inth$. 
Define the fixed subring of the $H$-action to be
$$A^H=\{ a\in A\mid h\cdot (a)=\epsilon(h) a, \forall h\in H\}$$
where $\epsilon$ is the counit of $H$.
For any algebra $A$ with $H$-action, there is a natural algebra 
homomorphism $\phi: A\# H\to \End_{A^H}(A)$ which sends $a\# h$ to 
an $A^H$-endomorphism of $A$:
$$\phi(a\# h): x\mapsto a (h\cdot (x)), \;\; \forall \; x \in A.$$
By \cite[Theorem 0.3]{BHZ1}, if $A$ is a noetherian, connected graded, 
Artin-Schelter regular and Cohen-Macaulay domain of $\GKdim \geq 2$,
then $\Pty(A,H)\geq 2$ if and only if the canonical map
\begin{equation}
\label{E0.0.4}\tag{E0.0.4}
\phi: \quad A\# H\longrightarrow \End_{A^H}(A)
\end{equation}
is an isomorphism. For simplicity, if $\phi$ is an isomorphism, 
we say that $(A,H)$ has the {\it isom-property}.

In this paper we are interested in the case when $H$ is 
$\Bbbk^G:=\Hom_{\Bbbk}(\Bbbk G, \Bbbk)$, or equivalently, $K$ is the 
group algebra $\Bbbk G$ for some finite group $G$, and when $A$ is 
the noetherian graded down-up algebra $\mathbb{D}(\alpha, \beta)$. Our main result 
is

\begin{theorem}
\label{xxthm0.1}  
Let $H:=\Bbbk^G$ act on $A:=\mathbb{D}(\alpha,\beta)$ homogeneously and 
inner faithfully, where $\beta\neq 0$. If the action has trivial 
homological determinant, then the pertinency $\Pty(A, H)\geq 2$. As a 
consequence, Auslander's theorem holds, namely, there is a 
natural isomorphism of graded algebras
$$\phi: \quad A\# H\cong \End_{A^H}(A).$$
\end{theorem}

Theorem \ref{xxthm0.1} fails without the hypothesis of ``trivial 
homological determinant'', see Remark \ref{xxrem1.6}(2).  
Theorem \ref{xxthm0.1} suggests there is a McKay correspondence for 
down-up algebras $\mathbb{D}(\alpha,\beta)$; it follows from 
\cite[Theorem A]{CKWZ2} that when Auslander's theorem holds, there 
are bijections between several categories of modules, e.g. simple 
left $H$-modules and indecomposable direct summands of $A$ as a 
left $A^H$-modules.  The paper \cite{QWZ} shows that whenever 
Auslander's theorem holds one can view $A\# H$ as a generalized 
noncommutative crepant resolution (NCCR) of $A^H$, and when 
$A^H$ is a central subalgebra of $A\# H$, $A\# H$ is an NCCR of 
$A^H$. 

The paper is organized as follows:  Section 1 contains some preliminary 
results, Section 2 contains the proof of Theorem \ref{xxthm0.1}, and 
Section 3 contains some examples.


\section{Preliminaries}
\label{xxsec1}
In this section we recall some basic definitions and make some comments.
We will omit the definition of Artin-Schelter Gorensteinness and
Artin-Schelter regularity \cite{AS} since these can be found in many 
other papers and we will not need these in the proof of the main result. 
As mentioned in the introduction, noetherian graded down-up algebras
are Artin-Schelter regular of global dimension three.

We introduce a temporary concept. For a graded module $C$ and an integer
$w$, the $w$th shift of $C$, denoted by $C(w)$, is defined by
$C(w)_m = C_{w+m}$ for all $m\in {\mathbb Z}$. 

\begin{definition}
\label{xxdef1.1} Let $H$ be a semisimple Hopf algebra acting on a 
connected graded algebra $A$ homogeneously and inner faithfully. 
Decompose $A$ into 
\begin{equation}
\label{E1.1.1}\tag{E1.1.1}
(\bigoplus_{i=1}^s A^H(-w_i)) \bigoplus B
\end{equation} 
as a right $A^H$-module 
for some integer $s\geq 1$, where $B$ has no direct summand that is
isomorphic to $A^H(w)$ for some integer $w$. If $B=0$, $H$ is called 
a {\it reflection Hopf algebra} with respect to $A$.  
If $s\geq 2$ (but $B\neq 0$), we say $H$ is a 
{\it fractional-reflection Hopf algebra} with respect to $A$, 
since part (but not all) of $A$ is a graded free $A^H$-module.
\end{definition}

\begin{lemma}
\label{xxlem1.2} 
Suppose $H$ acts on a connected graded algebra $A$ as a reflection
{\rm{(}}or fractional-reflection{\rm{)}} Hopf algebra.  Then:
\begin{enumerate}
\item[(1)]
$(A,H)$ does not have the isom-property.
\item[(2)]
If $A$ is a noetherian Artin-Schelter Gorenstein algebra, 
then the $H$-action on $A$ does not have trivial homological 
determinant.
\end{enumerate}
\end{lemma}

\begin{proof} (1) Since $H$ is a fractional-reflection Hopf algebra, 
$s\geq 2$ in \eqref{E1.1.1}. We write $A=A^H\oplus A^H(-w_2) \oplus C$ 
where $C$ is a right $A^H$-module. Note that $w_2$ is necessarily 
positive since $A$ is connected graded. There is a homogeneous 
$A^H$-module map of degree $-w_2$:
$$
A\xrightarrow{pr_{_{A^{H}(-w_2)}}} A^H(-w_2) 
\xrightarrow{shift\;  by\; degree\; w_2} A^H 
\xrightarrow{inclusion} A.
$$
Then $\End_{A^H}(A)$ has a nonzero element of negative degree. On the 
other hand, every nonzero homogeneous element in $A\# H$ has 
nonnegative degree. Therefore $A\# H\not\cong\End_{A^H}(A)$.

(2) We now assume that $A$ is noetherian and Artin-Schelter 
Gorenstein. If the $H$-action on $A$ has trivial homological 
determinant, then, by \cite[Theorem 3.6]{KKZ1}
and the proof of \cite[Lemma 3.5(d)]{KKZ1}, we have
\begin{enumerate}
\item[(a)]
$A^H$ is noetherian and Artin-Schelter Gorenstein, 
\item[(b)]
$\injdim A=\injdim A^H=:d$, and 
\item[(c)]
the AS indices of $A$ and $A^H$ are the same, denoted by $\ell$. 
\end{enumerate}
Let $\mathfrak{m}$ be the graded maximal ideal of $A^H$. We consider the local 
cohomology $R^d\Gamma_{\mathfrak{m}}(A)^*$ as in \cite{AZ, KKZ1}. Since $A^H(-w_2)$ is 
a direct summand of $A$ (as a right $A^H$-module), 
$R^d\Gamma_{\mathfrak{m}}(A^H(-w_2))^*$ is a direct summand of $R^d\Gamma_{\mathfrak{m}}(A)^*$. 
If both $A$ and $A^H$ are Artin-Schelter Gorenstein, by \cite[Lemma 3.5]{RRZ},
$$R^d\Gamma_{\mathfrak{m}}(A)^*\cong A(-\ell)\quad {\text{and}}
\quad R^d\Gamma_{\mathfrak{m}}(A^H(-w_2))^*
\cong A^H(-\ell+w_2).$$ The lowest degree of
nonzero element in $R^d\Gamma_{\mathfrak{m}}(A^H(-w_2))^*$ is $\ell-w_2$ 
and the lowest degree of
nonzero element in $R^d\Gamma_{\mathfrak{m}}(A)^*$ is $\ell$.
Since $w_2$ is positive, this is impossible. Therefore the $H$-action on $A$
does not have trivial homological determinant.
\end{proof}

\begin{remark}
\label{xxrem1.3} Lemma \ref{xxlem1.2}(2) is a generalization 
of \cite[Theorem 2.3]{CKWZ1}.
\end{remark}

The definition of maximal Cohen-Macaulay modules was extended to this context 
in \cite[Definition 3.5]{CKWZ3}.

\begin{proposition}
\label{xxpro1.4}
Let $A$ be connected graded and suppose that $(A, \Bbbk^G)$ has the 
isom-property. Write $A=\bigoplus_{g\in G} A_g$.  If $g\neq h$,
then $A_g$ is not isomorphic to $A_h(w)$ for any $w\in {\mathbb Z}$. 
As a consequence, if $A$ is noetherian Artin-Schelter Gorenstein, 
there are at least $|G|$ non-isomorphic graded maximal 
Cohen-Macaulay modules over $A^{co\; G}$, up to degree shift.
\end{proposition}

\begin{proof} Let $B=A^{co\; G}$. Suppose to the contrary that 
$A_g\cong A_h(w)$ for some $g\neq h$. If $w\neq 0$, then 
$\End_{B}(A)$ has an element of negative degree. So $A\# 
\Bbbk^G\not\cong \End_{B}(A)$, a contradiction. If $w=0$, 
then the degree zero part of $\End_{B}(A)$ contains a 
$2\times 2$ matrix algebra which is not commutative. However 
the degree 0 part of $A\# \Bbbk^G$ is $\Bbbk^G$, which is 
commutative. Therefore $A\# \Bbbk^G\not\cong \End_{B}(A)$, a 
contradiction. Therefore $A_g$ is not isomorphic to $A_h(w)$ 
if $g\neq h$.

The consequence is clear.
\end{proof}

The homological (co)determinant is defined in \cite{KKZ1}. 
We need some facts about the homological (co)determinant
of group coactions on down-up algebras. 
Suppose that $\mathbb{D}_{\beta}$ is $G$-graded with $\deg_G d=g_1$ and 
$\deg_G u=g_2$ (or equivalently, $G$ coacts on $\mathbb{D}_{\beta}$).
Assume that the $G$-coaction on $\mathbb{D}_{\beta}$ is inner-faithful, which 
is equivalent to the condition that $G$ is generated by $g_1$ and 
$g_2$, in this case.

\begin{lemma}
\label{xxlem1.5}
Retain the above notation.
The homological {\rm{(}}co{\rm{)}}determinant of the $\Bbbk^G$-action 
{\rm{(}}or $G$-coaction{\rm{)}} on $\mathbb{D}_{\beta}$ is $g_1^2 g_2^2$, 
and is trivial if and only if $g_1^2 g_2^2=1$, where $1$ is the unit 
of $G$.
\end{lemma}

\begin{proof} 
Let $A=\mathbb{D}_{\beta}$. Since $G$ coacts on $A$ homogeneously, 
$A$ is a ${\mathbb Z}\times G$-graded algebra. Recall that 
$\mathbb{D}_{\beta}$ is generated by $d$ and $u$ subject to relations
$$d^2 u = \beta u d^2, \quad 
d u^2 = \beta u^2 d.$$
By using the generators and relations 
of $A$, one checks that the $G$-graded resolution of the trivial 
$A$-module $\Bbbk$ is 
{\small $$0\to A(g_1^{-2}g_2^{-2}) 
\to A(g_1^{-1}g_2^{-2})\oplus A(g_1^{-2}g_2^{-1})
\to A(g_1^{-1})\oplus A(g_2^{-1})\to A\to\Bbbk \to 0.$$}

\noindent
Using this resolution to compute the $\Ext$-group, one sees that
$\Ext^3_{A}(\Bbbk,\Bbbk)\cong \Bbbk (g_1^2g_2^2)$ as a $G$-graded
vector space. Hence the $G$-coaction maps a basis element
${\mathfrak e}\in \Ext^3_{A}(\Bbbk,\Bbbk)$ to ${\mathfrak e}
\otimes g_1^2g_2^2$. By definition, the homological
codeterminant of the $G$-coaction is $g_1^2g_2^2$. The assertion 
follows.
\end{proof} 

Next we make some comments about 
\cite[Example 2.1]{CKZ}.

\begin{remark}
\label{xxrem1.6} Consider the algebra $\mathbb{D}:=\mathbb{D}_{1}$ as in 
\cite[Example 2.1]{CKZ}. 
\begin{enumerate}
\item[(1)]
By \cite[Theorem 0.6]{BHZ2}, if $H=\Bbbk G$ for any finite group
$G$ acting on $\mathbb{D}$, then $\Pty(\mathbb{D},G)\geq 2$ and 
$\mathbb{D}\ast G\cong \End_{\mathbb{D}^G}(\mathbb{D})$, so that Auslander's theorem 
holds for group actions on $\mathbb{D}$; this result was expected because 
all finite groups acting on $\mathbb{D}$ are ``small'', since they have 
no reflections, in a sense made precise in \cite{KKZ2}.
But, when $H=\Bbbk^G$ as in \cite[Example 2.1 and Lemma 2.2]{CKZ},
$H$ is a fractional-reflection Hopf algebra with respect to $\mathbb{D}$, so
by Lemma \ref{xxlem1.2}(1), $(\mathbb{D},H)$ does not have the isom-property,
namely, Auslander's theorem fails. By \cite[Theorem 0.3]{BHZ1}, 
$\Pty(\mathbb{D},\Bbbk^G)\leq 1$, and  by Lemma \ref{xxlem1.2}(2), the 
$\Bbbk^G$-action does not have trivial homological determinant.  
Hence group actions behave differently from group coactions.
\item[(2)]
By \cite[Corollary 4.11]{KKZ1} or \cite[Proposition 0.2(2)]{KKZ2}, 
if $H=\Bbbk G$ for some finite group $G$, then $\mathbb{D}^G$ is Gorenstein 
if and only if the $G$-action on $\mathbb{D}$ has trivial homological 
determinant. This result was expected because, again, these groups 
contain no reflections of $\mathbb{D}$. However, \cite[Example 2.1]{CKZ} 
shows that when $H=\Bbbk^G$, this statement fails, namely, $\mathbb{D}^{co\; G}$ 
is Gorenstein, but the $\Bbbk^G$-action does not have trivial homological 
determinant. This result is surprising, and there might be a 
relationship between the facts in parts (1) and (2).
\item[(3)]
Theorem \ref{xxthm0.1} implies that if the $\Bbbk^G$-action on $\mathbb{D}$ 
has trivial homological determinant, then $\Pty(\mathbb{D},\Bbbk^G)\geq 2$ 
and $\mathbb{D}\# \Bbbk^G\cong \End_{\mathbb{D}^{co\; G}}(\mathbb{D})$.
\item[(4)] 
In the commutative case, when a semisimple Hopf algebra acts on a 
polynomial ring $A:=\Bbbk[x_1,\cdots,x_n]$, Auslander's theorem 
fails if and only if there is a nontrivial Hopf subalgebra 
$H_0\subseteq H$ (in this case, $H$ and $H_0$ are group algebras 
$\Bbbk G$ and $\Bbbk G_0$ respectively for some $G_0\subseteq G$) 
such that $A^{H_0}$ is Artin-Schelter regular; this happens if and 
only if $G$ contains a reflection, or equivalently, $G$ is not small. 
Recall that a finite subgroup $G$ of ${\text{GL}}_n(\Bbbk)$ is 
{\it small} if it does not contain any 
reflections. Hence one might conjecture that Auslander's holds for a 
semisimple Hopf algebra if and only if there is no such Hopf subalgebra, 
and that this definition is the generalization for Hopf algebras of 
the notion of a ``small group''. However, \cite[Example 2.1]{CKZ},  
where $H=\Bbbk^G$, shows that this definition of an analogue of a 
``small subgroup'' does not work, since in this example Auslander's 
Theorem fails, but as one can easily check, or use 
\cite[Proposition 0.2(2)]{KKZ2}, that there is NO nontrivial Hopf 
subalgebra $H_0\subseteq H$ such that $\mathbb{D}^{H_0}$ is Artin-Schelter 
regular. So it is not clear how to generalize Auslander's theorem 
beyond our noncommutative analogue of subgroups of ${\rm{SL}}_n(\Bbbk)$ 
(namely $H$-actions with trivial homological determinant) to a 
noncommutative analogue for Hopf algebras of the notion of ``small'' 
groups (groups containing no reflections). 
\end{enumerate}
\end{remark}

\begin{question} 
\label{xxque1.7}
For actions (and coactions) by semisimple Hopf algebras $H$ on 
Artin-Schelter regular algebras $A$, is there an analogue of the 
action  on $\Bbbk[x_1, \dots, x_n]$ by a finite ``small'' 
subgroup of GL$_n(\Bbbk)$ (a condition for Hopf actions with 
non-trivial homological determinant for which Auslander's 
Theorem holds)?
\end{question}

To prove Theorem \ref{xxthm0.1}, we only need to show that 
$\Pty(\mathbb{D}(\alpha,\beta), \Bbbk^G)\geq 2$. The pertinency 
$\Pty(A,H)$ is defined in \eqref{E0.0.3}.

Let $\inth$ be the integral of a semisimple Hopf algebra $H$, 
and $I$ be the two-sided ideal of $A \# H$ generated by $1 \# \inth$.
Recall that a $\Bbbk^G$-action on an algebra $A$ is equivalent 
to a $G$-grading on $A$. 

We recall the following result from \cite{BHZ1} that will be used in the
pertinency computation. 

\begin{lemma}  
\label{xxlem1.8} 
Let $H := (kG)^\circ$ act on $A$ inner
faithfully, and write $A = \oplus_{g\in G} Ag$.
\begin{enumerate}
\item  
\cite[Lemma 5.1 (3)]{BHZ1} 
If $f \in \cap_{g \in G} AA_g$ then $f \# 1 \in I$.
\item  
\cite[Lemma 5.1 (6)]{BHZ1} 
$$\Pty(R,H) \geq  d - \GKdim A/( \cap_{g\in G} AA_g) 
\geq  d -\max\{\GKdim A/AA_g | g \in G \}.$$
\end{enumerate}
\end{lemma}

The following is a modification of 
\cite[Lemma 5.1(4)]{BHZ1}.

\begin{lemma}
\label{xxlem1.9} 
Let $G$ be a finite group and $\Bbbk^G$ act on $A$ 
inner-faithfully and homogeneously. Let $z\in A$. 
\begin{enumerate}
\item[(1)]
Suppose that, for each $g\in G$, there is an $x\in A$ 
of $G$-degree $g$ and $y\in A$ such that $z=yx$. Then 
$z\# 1$ is in the ideal of $A\# \Bbbk^G$ generated by $e:=1\#\inth$.
\item[(2)] 
Suppose $z=f_n\cdots f_1$ is such that the collection {\rm{(}}with 
possible repetitions{\rm{)}}
$$\{1, \deg_G (f_1), \deg_G (f_2f_1), 
\cdots, \deg_G(f_{n-1}\cdots f_2f_1), \deg_G (z)\}$$ 
includes all elements in $G$. Then  $z\# 1$ is in the ideal of 
$A\# \Bbbk^G$ generated by $e:=1\#\inth$.
\end{enumerate}
\end{lemma}

\begin{proof} (1) Since $z=yx\in A A_{g}$ for each $g$, we have
$z\in \bigcap_{g\in G} A A_g$. The assertion follows 
from Lemma \ref{xxlem1.8}.

(2) This is a special case of part (1).
\end{proof}

In the next lemma we use some arguments from Bergman's Diamond Lemma 
\cite{Be}. Recall that $\mathbb{D}(\alpha, \beta)$ is generated by $d$ and $u$.
We use the ordering $d<u$ in this paper. Two relations of 
$\mathbb{D}(\alpha, \beta)$, namely, \eqref{E0.0.1}-\eqref{E0.0.2} can be 
written as 
$$\begin{aligned}
u d^2&= {\rm{lower}}\;\; {\rm{terms}},\\
u^2 d&= {\rm{lower}}\;\; {\rm{terms}}
\end{aligned}
$$
where ``${\rm{lower}}\;\; {\rm{terms}}$'' stands for a linear combination 
of monomials that have lower degree (in the lexicographic order) than 
the terms explicitly appearing in the same equation. 

\begin{lemma}
\label{xxlem1.10} Retain the above notation.
\begin{enumerate}
\item[(1)]
Let $W$ be an ideal of $\mathbb{D}(\alpha, \beta)$ such that, in the factor
ring $\mathbb{D}(\alpha, \beta)/W$, there are relations
$$\begin{aligned}
d^{s} (ud)^i&= {\rm{lower}}\;\; {\rm{terms}},\\
u^t &={\rm{lower}}\;\; {\rm{terms}}
\end{aligned}
$$
for some $i,s,t\geq 0$. 
Then $\GKdim \mathbb{D}(\alpha,\beta)/W\leq 1$.
\item[(2)]
Let $W$ be an ideal of $\mathbb{D}_{\beta}$ such that, in the factor
ring $\mathbb{D}_{\beta}/W$, there are relations
$$\begin{aligned}
d^{2s} (du)^i&= {\rm{lower}}\;\; {\rm{terms}},\\
(ud)^j u^{2t} &={\rm{lower}}\;\; {\rm{terms}}
\end{aligned}
$$
for some $i,j,s,t\geq 0$. Then $\GKdim \mathbb{D}_{\beta}/W\leq 1$.
\end{enumerate}
\end{lemma}

\begin{proof} (1) Together with \eqref{E0.0.1}-\eqref{E0.0.2},
we have at least four relations
$$\begin{aligned}
u d^2&= {\rm{lower}}\;\; {\rm{terms}},\\
u^2 d&= {\rm{lower}}\;\; {\rm{terms}},\\
d^{s} (ud)^i&= {\rm{lower}}\;\; {\rm{terms}},\\
u^t &={\rm{lower}}\;\; {\rm{terms}}
\end{aligned}
$$
in the factor ring $\mathbb{D}(\alpha, \beta)/W$. By the Diamond 
Lemma \cite{Be} and using the first two relations, 
$\mathbb{D}(\alpha,\beta)/W$ has a $\Bbbk$-linear basis 
consisting of monomials of the form
$$ d^a (ud)^b u^c, \quad a,b,c\geq 0$$
with some constraints. (A similar statement is 
\cite[Lemma 1.1(3)]{CKZ} where we use the order $u<d$.) 
Two of the constraints are 
(i) either $a<s$ or $b<i$ and (ii) $c<t$, which follows from 
the last two relations of $\mathbb{D}(\alpha,\beta)/W$. Therefore, 
for each ${\mathbb N}$-degree $d$, the $\Bbbk$-dimension of 
$(\mathbb{D}(\alpha,\beta)/W)_d$ is uniformly bounded. As a 
consequence of a Gelfand-Kirillov dimension computation 
\cite[(E1.1.6)]{BHZ2}, $\GKdim \mathbb{D}(\alpha,\beta)/W\leq 1$.

(2) The proof is similar to the one of part (1) and uses the fact that
$d^2$ and $u^2$ are normal elements of $\mathbb{D}_\beta$.

Without loss of generality, we can assume that $s=t=i=j=: a>0$
and re-use the letters $i$ and $j$. Let 
$$\begin{aligned}
d^{2a}(du)^a &= {\rm{lower}}\;\; {\rm{terms}},\\
(ud)^a u^{2a} &= {\rm{lower}}\;\; {\rm{terms}}
\end{aligned}
$$
in $\mathbb{D}_{\beta}/W$. Then 
$$d^{4a}u^{4a}= \lambda (d^{2a}(du)^a)( (ud)^a u^{2a})
={\rm{lower}}\;\; {\rm{terms}}$$
in $\mathbb{D}_{\beta}/W$, for some $\lambda\in \Bbbk$. 
Then $d^{4a} u^{4a} = d^{4a+1} f$ in $\mathbb{D}_{\beta}/W$ for some
$f$. Since $u^2$ skew-commutes with $d$ and $u$, we obtain that 
$$d^{4a} (ud)^j u^{4a} = d^{4a+1} f' $$
or 
$$d^{4a} (ud)^j u^{4a} = {\rm{lower}}\;\; {\rm{terms}}.$$
Therefore 
$$d^i(ud)^ju^k= {\rm{lower}}\;\; {\rm{terms}}$$ 
in $\mathbb{D}_{\beta}/W$ when at least two of indices $i, j, k$ 
are larger than $4a$. By the Diamond Lemma argument as in 
the proof of part (1), for each ${\mathbb N}$-degree
$d$, the $\Bbbk$-dimension of $(\mathbb{D}_{\beta}/W)_d$ is uniformly 
bounded. By \cite[(E1.1.6)]{BHZ2}, $\GKdim \mathbb{D}_{\beta}/W\leq 1$.
\end{proof}

\section{Proof of Theorem \ref{xxthm0.1}}
\label{xxsec2}

In this section we prove the main result, the theorem of
Auslander for group coactions on down-up algebras. First 
we recall a result from \cite{CKZ}.

Let ${\mathbb{F}}$ be the algebra generated by $x$ and $y$,  subject to 
two relations
\begin{equation}
\label{E2.0.1}\tag{E2.0.1}
x^3=yxy \quad \mbox{ and } \quad y^3=xyx.
\end{equation}
As a graded algebra, ${\mathbb{F}}$ is isomorphic to $\mathbb{D}_{-1}$
\cite[Lemma 1.5(1)]{CKZ}. Let ${\mathbb{H}}$ be the algebra 
generated by $x$ and $y$,  subject to two relations
\begin{equation}
\label{E2.0.2}\tag{E2.0.2}
x^2 y+yx^2 -2y^3=0 \quad \mbox{ and } \quad 
-2x^3+xy^2+y^2 x=0.
\end{equation}
Then, as a graded algebra, ${\mathbb{H}}$ is isomorphic 
to $\mathbb{D}(-2,-1)$ \cite[Lemma 1.9(1)]{CKZ}.

\begin{lemma} \cite[Proposition 1.12]{CKZ}
\label{xxlem2.1}
Suppose $G$ is a finite non-cyclic group coacting on
$A:=\mathbb{D}(\alpha,\beta)$ homogeneously and inner-faithfully.
Then one of the following occurs.
\begin{enumerate}
\item[(1)]
$\alpha=0$ and $u$ and $d$ are $G$-homogeneous after
a change of variables.
\item[(2)]
$A$ is isomorphic to ${\mathbb{F}}$ and using the generators of ${\mathbb{F}}$,
both $x$ and $y$ are $G$-homogeneous.
\item[(3)]
$A$ is isomorphic to ${\mathbb{H}}$ and using the generators of ${\mathbb{H}}$,
both $x$ and $y$ are $G$-homogeneous.
\item[(4)]
$G$ is abelian and there are linearly independent elements
$x$ and $y$ of $\mathbb{D}(\alpha, -1)$ of degree one
such that
$$\begin{aligned}
\alpha x^2 y+(-2-\alpha) xyx+ \alpha yx^2 +(2-\alpha) y^3&=0,\\
(2-\alpha) x^3+ \alpha xy^2+(-2-\alpha) yxy+ \alpha y^2 x&=0
\end{aligned}
$$
and  $x$ and $y$ are $G$-homogeneous.
\item[(5)]
$G$ is abelian and $u$ and $d$ are $G$-homogeneous after a change
of variables.
\end{enumerate}
\end{lemma}

The above lemma shows that there are plenty of interesting examples of 
finite group coactions on noetherian down-up algebras.

Note that the hypothesis of ``$G$ being non-cyclic'' is needed in 
the above lemma which was proved in \cite{CKZ}. In the present
paper we will also consider cyclic cases. In particular, our main 
theorem does not need the hypothesis of ``$G$ being non-cyclic''.

We separate the proof of Theorem \ref{xxthm0.1} into subcases 
according to the above lemma. In Cases 1 and 2 we assume that 
$G$ is not cyclic; the cyclic cases will be included in Case 3. 

\subsection{Case 1: $\alpha=0$, $u$ and $d$ are $G$-homogeneous}
\label{xxsec2.1}
In this subsection, as $\alpha=0$, $A$ is the down-up algebra 
$$\mathbb{D}_{\beta}:=\Bbbk\langle d,u \rangle/(d^2u-\beta ud^2, du^2-\beta u^2d),\;\;
\beta\in \Bbbk^{\times}.$$
Suppose that $\mathbb{D}_{\beta}$ is $G$-graded with $\deg_G d=g_1$ and 
$\deg_G u=g_2$. Since $\mathbb{D}_{\beta}$ is generated by $d$ and $u$, 
$G$ is generated by $g_1$ and $g_2$. Let 
\begin{equation}
\label{E2.1.1}\tag{E2.1.1}
X_1:=\{(g_2g_1)^i\mid i\geq 0\}\cup \{g_1(g_2g_1)^i\mid i\geq 0\}\subseteq G
\end{equation}
and 
\begin{equation}
\label{E2.1.2}\tag{E2.1.2}
X_2:=\{(g_1g_2)^i\mid i\geq 0\}\cup \{g_2(g_1g_2)^i\mid i\geq 0\}\subseteq G.
\end{equation}
As in \cite[Lemma 5.1]{BHZ1} let
\begin{equation}
\label{E2.1.3}\tag{E2.1.3}
J:={\rm{the}}\;\; {\rm{ideal}}\;\; {\rm{of}}\;\; A \;\;
{\rm{generated}}\;\; {\rm{by}}\;\; \bigcap_{g\in G} A A_g 
\end{equation}
when a group $G$ coacts on $A$.

\begin{lemma}
\label{xxlem2.2}
Suppose that $\langle g_1\rangle X_1=G=\langle g_2\rangle X_2$.
Then $\Pty(\mathbb{D}_{\beta},\Bbbk^G)\geq 2$. 
\end{lemma}

\begin{proof} Let $A=\mathbb{D}_{\beta}$. By Lemma \ref{xxlem1.8} (2)
it suffices to show that
$$\GKdim A/J\leq 1$$
where $J$ is defined as in \eqref{E2.1.3}. By Lemma \ref{xxlem1.10}(2), 
it suffices to show that $v:=d^{2a} (du)^{a+1}$ and $w:=u^{2a} (ud)^a$ 
are in the ideal $J$, where $a=|G|$. By symmetry, we show only that 
$v$ is in $J$. 

Since $v=d d^{2a}(ud)^a u$, it suffices to show that $f:=d^{2a} 
(ud)^a$ is in $J$. By hypothesis $\langle g_1\rangle X_1=G$, every 
element $g$ in $G$ is of the form  $g_1^i (g_2g_1)^j$ for some 
$a\geq i,j\geq 0$. Since $d^2$ is normal,  we can write $f$ as 
$c(ud)^{a-j} d^{2a-i} \cdot d^i (ud)^j$, for some $c\in \Bbbk^{\times}$ 
with $\deg_G d^i (du)^j=g_1^i (g_2g_1)^j=g$. Then $f\in A A_g$ for all 
$g$, which implies that $f\in J$ as required.
\end{proof}

\begin{lemma}
\label{xxlem2.3} Suppose $G$ is generated by $g_1,g_2$ and 
$\langle g_1^2\rangle =\langle g_2^2\rangle$. Then 
$\langle g_1\rangle X_1=G=\langle g_2\rangle X_2$.
\end{lemma}

\begin{proof} Let $N$ be the normal subgroup of $G$ generated by 
$g_1^2$ and $g_2^2$. Then $G/N$ is a dihedral group $D_{2n}$. In 
this case the image of $X_1$ in $G/N$ consists of all elements in 
$G/N$. Then $G=N X_1$. Under the hypothesis, we have 
$N=\langle g_1^2\rangle$. Hence $\langle g_1\rangle X_1=G$. By 
symmetry, $G=\langle g_2\rangle X_2$.
\end{proof}

Now we are ready to prove a part of Theorem \ref{xxthm0.1}.

\begin{proposition}
\label{xxpro2.4}
Retain the notation as in Theorem {\rm{\ref{xxthm0.1}}}. Suppose further 
that $\alpha=0$ and $u$ and $d$ are $G$-homogeneous. Then 
$\Pty(A,H)\geq 2$.
\end{proposition}

\begin{proof}
By Lemma \ref{xxlem1.5}, when the $\Bbbk^G$-action on $\mathbb{D}_{\beta}$ has 
trivial homological determinant, $g_1^2=g_2^{-2}$. Hence 
$\langle g_1^2\rangle =\langle g_2^2\rangle$. By Lemma \ref{xxlem2.3},
$\langle g_1\rangle X_1=G=\langle g_2\rangle X_2$. Now the main 
assertion follows from Lemma \ref{xxlem2.2}. 
\end{proof}

\subsection{Case 2: $A={\mathbb{H}}$, $x$ and $y$ are $G$-homogeneous}
\label{xxsec2.2}
In this subsection we have that $A={\mathbb{H}}$ and that $x$ and $y$ in 
${\mathbb{H}}$ are $G$-homogeneous. Let $g_1=\deg_G x$ and $g_2=\deg_G y$. 
By the relations of ${\mathbb{H}}$, one sees that $g_1^2=g_2^2$. Two 
relations of ${\mathbb{H}}$ can be written as 
$$x(x^2-y^2)=-(x^2-y^2)x\quad {\rm{and}} \quad y(x^2-y^2)=-(x^2-y^2)y.$$
Define a filtration ${\mathcal F}$ on ${\mathbb{H}}$ by
$$F_i {\mathbb{H}}=(\Bbbk x+\Bbbk y+ \Bbbk z)^i, i\geq 0$$
where $z=x^2-y^2$. It is easy to see that the $G$-coaction
preserves this filtration. Let $B$ be gr$_{\mathcal F} {\mathbb{H}}$.
Then $B\cong (\Bbbk\langle x,y\rangle/(x^2-y^2))[z,\sigma]$
where $\sigma$ maps $x\to -x$ and $y\to -y$. Then $G$ coacts
on $B$ by $\deg_G x=g_1, \deg_G y=g_2$ and $\deg_G z=g_1^2$.
The following lemma follows from \cite[Lemma 3.6]{BHZ1}.

\begin{lemma}
\label{xxlem2.5} Retain the above notation. 
Then $\Pty({\mathbb{H}}, \Bbbk^G)\geq \Pty (B, \Bbbk^G)$.
\end{lemma}

By the above lemma, it suffices to show that $\Pty (B, \Bbbk^G)
\geq 2$. For the rest of the proof we follow the proof in Case 1.

\begin{lemma}
\label{xxlem2.6} 
Let $J$ be an ideal of $B$  containing both 
$x^{2s} (yx)^i$ and $(xy)^j z^t$ for some $i,j,s,t\geq 0$. 
Then $\GKdim B/J\leq 1$.
\end{lemma}

\begin{proof} 
Without loss of generality, we can assume that $s=t=i=j=:a>0$ 
and re-use letters $i$ and $j$. Let $f_1=x^{2a} (yx)^a$ and 
$f_2=(xy)^a z^a$ in $J$. Then $x^{6a} z^{a}\in \Bbbk f_1f_2 
\subseteq J$. Note that $B$ has a $\Bbbk$-linear basis 
$$\{x^i (yx)^j z^k\mid i,j,k \geq 0\}\cup
\{x^i (yx)^j z^k y\mid i,j,k \geq 0\}.$$
Since $x^2(=y^2)$ and $z$ are skew-commuting with $x,y,z$, 
every element is of the form $x^i (yx)^j z^k$ or
$x^i (yx)^j z^k y$ is 0 in $B/J$ when at least two of 
indices $i,j,k$ are larger than $6a$. An elementary counting argument 
shows that $\GKdim B/J\leq 1$. 
\end{proof}

Use the notation introduced in \eqref{E2.1.1} and 
\eqref{E2.1.2}:
$$X_1:=\{(g_2g_1)^i\mid i\geq 0\}\cup \{g_1(g_2g_1)^i\mid i\geq 0\}\subseteq G$$
and 
$$X_2:=\{(g_1g_2)^i\mid i\geq 0\}\cup \{g_2(g_1g_2)^i\mid i\geq 0\}\subseteq G.$$ 

\begin{lemma}
\label{xxlem2.7} Retain the above notation.
\begin{enumerate}
\item[(1)]
$\langle g_1^2\rangle X_1=G=\langle g_2^2\rangle X_2$.
\item[(2)]
$\Pty(B,\Bbbk^G)\geq 2$.
\item[(3)]
$\Pty({\mathbb{H}}, \Bbbk^G)\geq 2$.  
\end{enumerate}
\end{lemma}

\begin{proof} (1) Let $N$ be the normal subgroup of $G$ 
generated by $g_1^2$ (or by $g_2^2$). Then $G/N$ is a 
dihedral group $D_{2n}$. In this case the image of $X_1$ 
in $G/N$ consists of all elements in $G/N$. Then 
$G=N X_1=\langle g_1^2\rangle X_1$. 
Similarly, $G=\langle g_2^2\rangle X_2$.

(2) By 
Lemma \ref{xxlem1.8}(2) it suffices to show that
$$\GKdim B/J\leq 1$$
where $J$ is the ideal of $B$ generated by $\bigcap_{g\in G} B B_g$.
By Lemma \ref{xxlem2.6}, it suffices to show that $f_1:=x^{2a} (yx)^a$ 
and $f_2:=(xy)^a z^a$ are in the ideal $J$, where $a=|G|$. By part (1), 
$\langle g_1\rangle X_1=G$, every element in $G$ is of the form 
$g_1^i (g_2g_1)^j$ for some $0\leq i,j\leq a$. By the fact that 
$x^2$ commutes with $y$, we obtain that $f_1=f_1' (x^i (yx)^j)$ 
for some $f_1'\in B$. Then $f_1\in B B_{g}$ for all $g$, which 
implies that $f_1\in J$. Since $z$ skew-commutes with $x$ and $y$, 
a similar argument shows that $f_2\in J$. Now the assertion 
follows by Lemma \ref{xxlem2.6}.

(3) This follows from part (2) and Lemma \ref{xxlem2.5}. 
\end{proof}

Part (3) of the above lemma says that Auslander's theorem 
holds in this case, even without the hypothesis that the homological
determinant of the $H$-action is trivial in this special case. 
For the sake of completeness we calculate the homological
(co)determinants of the $G$-coactions easily in the next lemma.

\begin{lemma}
\label{xxlem2.8} Suppose a finite group $G$ coacts on $A$.
\begin{enumerate}
\item[(1)] 
If $A={\mathbb{H}}$ and $x$ and $y$ are $G$-homogeneous with $G$-degree 
$g_1$ and $g_2$ respectively, then the homological codeterminant of the 
$G$-coaction is $g_1^4$, which is also $g_2^4$.
\item[(2)] 
If $A={\mathbb{F}}$ and $x$ and $y$ are $G$-homogeneous with $G$-degree 
$g_1$ and $g_2$ respectively, then the homological codeterminant of 
the $G$-coaction is $g_1^4$, which is also $g_2^4$.
\item[(3)]
Let $A=\mathbb{D}(\alpha, -1)$, for $\alpha\neq 2$, and $x=\frac{1}{2}(d+u)$ and 
$y=\frac{1}{2}(d-u)$. By \cite[Proposition 1.12(4)]{CKZ}, 
$A$ is generated by $x$ and $y$ and subject to relations
\begin{align}
\label{E2.8.1}\tag{E2.8.1}
\alpha x^2 y+(-2-\alpha) xyx+ \alpha yx^2 +(2-\alpha) y^3&=0,\\
\label{E2.8.2}\tag{E2.8.2}
(2-\alpha) x^3+ \alpha xy^2+(-2-\alpha) yxy+ \alpha y^2 x&=0.
\end{align}
Suppose that $G$ is abelian and that $x$ and $y$ are $G$-homogeneous 
with $G$-degree $g_1$ and $g_2$ respectively. Then the homological
codeterminant of the $G$-coaction is $g_1^4$, which is also $g_2^4$.
\end{enumerate}
\end{lemma}

\begin{proof} Since the proofs are similar to the proof of Lemma 
\ref{xxlem1.5}, the details are omitted.
\end{proof}

\subsection{Case 3: $G$ is abelian}
\label{xxsec2.3}
Let $G$ be a finite abelian group and let $\widehat{G}$ 
be the character group $\Hom_{groups} (G, \Bbbk^{\times})$.
Since $\Bbbk$ is algebraically closed of characteristic zero,
$\widehat{G}$ is isomorphic to $G$ as an abstract group.
As a consequence, $(\Bbbk G)^{*}$ is isomorphic to
$\Bbbk G$ as a Hopf algebra.

Let $A$ be a down-up algebra $\mathbb{D}(\alpha,\beta)$ generated 
by $d$ and $u$. Every graded algebra automorphism $g$ of
$A$ can be written as a $2\times 2$-matrix with respect to
the basis $\{d,u\}$. We say $g$ is {\it diagonal} (respectively, 
{\it non-diagonal}) if its matrix presentation with respect 
to $\{d,u\}$ is diagonal (respectively, non-diagonal).
When the basis $\{d,u\}$ is replaced by 
$\{d',u'\}=\{c_1 d, c_2 u\}$ for some $c_1,c_2\in \Bbbk^{\times}$, 
the matrix presentation of $g$ could change accordingly, but the
diagonal property of $g$ will not change. We call this kind of change of 
basis a {\it scalar base change} which we use in the proof of Lemma 
\ref{xxlem2.9}.

Let $G$ be a finite abelian group that coacts on $A$ inner-faithfully
and homogeneously. This $G$-coaction on $A$ is equivalent to a 
$\widehat{G}$-action on $A$ preserving the ${\mathbb N}$-grading. 
Therefore we can consider the $\widehat{G}$-action instead of the
$G$-coaction. The theorem of Auslander was proved for finite group 
actions on graded noetherian down-up algebras in 
\cite[Theorem 0.6]{BHZ2} except for the case $A=\mathbb{D}(\alpha, -1)$ 
for $\alpha\neq 2$. In fact their proof \cite[Proof of Theorem 0.6]{BHZ2}
works for any diagonal automorphisms of $\mathbb{D}(\alpha,-1)$, too, and 
\cite[Proposition 4.6]{HZ} handles another special class of groups 
acting on $A=\mathbb{D}(\alpha, -1)$ for $\alpha\neq 2$. In this subsection 
we prove Auslander's theorem only for a finite abelian group 
$\widehat{G}$ of graded automorphisms of $\mathbb{D}(\alpha,-1)$ with 
$\alpha\neq 2$ that is not all diagonal. Combining with the results 
in \cite{BHZ2}, we take care of all abelian groups (including cyclic 
ones). 

Throughout the rest of this subsection let $A$ be $\mathbb{D}(\alpha, -1)$ 
for some $\alpha\neq 2$. The next lemma classifies all possible 
finite abelian groups that are not diagonal having trivial homological
determinant.

\begin{lemma}
\label{xxlem2.9}
Consider the following subgroup of ${\rm{GL}}_2(\Bbbk)$ 
$$T=\left\{ 
\begin{pmatrix}a &0 \\ 0 & a\end{pmatrix},
\begin{pmatrix}0 &b \\ b & 0\end{pmatrix}: a,b\in \{\pm 1, \pm i\}\right\}.$$
The following hold.
\begin{enumerate}
\item[(1)]
$T$ is an abelian group acting naturally on $A$, with respect to the basis 
$\{d,u\}$,  inner-faithfully and homogeneously with trivial homological 
determinant. 
\item[(2)]
Let $\widehat{G}$ be a finite abelian group acting on $A$ 
inner-faithfully and homogeneously with trivial homological 
determinant. If $\widehat{G}$ contains a non-diagonal matrix,
then $\widehat{G}$ is a subgroup of $T$ after a scalar base change.
\item[(3)]
Let $\widehat{G}$ be as in part {\rm{(2)}}. Then, up to a scalar base change,
$\widehat{G}$ is one of the following.
$$\left\{ 
\begin{pmatrix}1 &0 \\ 0 & 1\end{pmatrix},
\begin{pmatrix}0 &1 \\ 1 & 0\end{pmatrix}\right\},\quad 
\left\{ 
\begin{pmatrix}a &0 \\ 0 & a\end{pmatrix},
\begin{pmatrix}0 &b \\ b & 0\end{pmatrix}: a,b\in \{\pm 1\}\right\},$$
$$\left\{ 
\begin{pmatrix}a &0 \\ 0 & a\end{pmatrix},
\begin{pmatrix}0 &b \\ b & 0\end{pmatrix}: a\in \{\pm 1\}, b\in \{\pm i\}\right\},
\quad {\rm{or}} \quad T.$$
\end{enumerate}
\end{lemma}

\begin{proof} (1) This follows by a direct computation.

(2) Suppose that $f:=\begin{pmatrix}a &0 \\ 0 & d\end{pmatrix}$
and $g:=\begin{pmatrix}0 &b \\ c & 0\end{pmatrix}$ are in $\widehat{G}$.
The commutativity of $G$ forces  $a=d$. By \cite[Theorem 1.5]{KK},
the homological determinant of $\begin{pmatrix}a &0 \\ 0 & a\end{pmatrix}$
is $a^4$. Thus $a\in \{\pm 1, \pm i\}$ as $\widehat{G}$ has trivial 
homological determinant. In other words, $f\in T$. After a scalar 
base change, we may assume that $b=c$ in the matrix $g$. By 
\cite[Theorem 1.5]{KK}, the homological determinant of $g$ (with $b=c$)
is $b^4$. Then $b\in \{\pm 1, \pm i\}$ and $g\in T$. Now assume that
$\widehat{G}$ contains another non-diagonal automorphism 
$h:=\begin{pmatrix}0 &c' \\ c & 0\end{pmatrix}$. Then 
the equation $gh=hg$ implies that $c'=c$. So $h\in T$ and $\widehat{G}$ is a 
subgroup of $T$. 

(3) This follows by a direct computation.
\end{proof}

Using the classification in Lemma \ref{xxlem2.9}, we can work out 
the corresponding coactions. Let $x=\frac{1}{2}(d+u)$ and 
$y=\frac{1}{2}(d-u)$, or equivalently, $d=x+y$ and $u=x-y$. By the 
proof of \cite[Proposition 1.12(4)]{CKZ}, we have the following.

\begin{lemma}
\label{xxlem2.10}
Suppose $G$ is a finite abelian group coacting on $A:=\mathbb{D}(\alpha,-1)$, 
for $\alpha\neq 2$, such that 
\begin{enumerate}
\item[(a)]
the $G$-coaction has trivial homological codeterminant, and 
\item[(b)]
the corresponding $\widehat{G}$-action contains a non-diagonal matrix 
with respect to the basis $\{d,u\}$. 
\end{enumerate}
Then the following hold.
\begin{enumerate}
\item[(1)]
There are linearly independent elements
$x$ and $y$ of $\mathbb{D}(\alpha, -1)$ of degree one
such that
$$\begin{aligned}
\alpha x^2 y+(-2-\alpha) xyx+ \alpha yx^2 +(2-\alpha) y^3&=0,\\
(2-\alpha) x^3+ \alpha xy^2+(-2-\alpha) yxy+ \alpha y^2 x&=0
\end{aligned}
$$
and  $x$ and $y$ are $G$-homogeneous.
\item[(2)]
Let $\deg_G x=g_1$, $\deg_G y=g_2$. Then $g_1^2=g_2^2$ and 
$g_1^4=g_2^4=1$ in $G$. 
\end{enumerate}
\end{lemma}

\begin{proof} (1) A part of the proof appeared in the proof of 
\cite[Proposition 1.12]{CKZ}, so we give only a sketch of the 
argument here.

First, the $\widehat{G}$-action on $A$ has the special forms 
as listed in Lemma \ref{xxlem2.9}(3). Using the forms given 
there, let $x=\frac{1}{2}(d+u)$ and $y=\frac{1}{2}(d-u)$, or 
$d=x+y$ and $u=x-y$. Then both $x$ and $y$ are $\widehat{G}$-eigenvectors.
This means that both $x$ and $y$ are $G$-homogeneous in the 
corresponding $G$-coaction. The two relations are obtained in the 
proof of \cite[Proposition 1.12]{CKZ} by direct computation, 
which we will not repeat here.

(2) By the relations and the hypothesis that $\alpha\neq 2$, 
one sees that $g_1^2=g_2^2$. The second assertion is Lemma 
\ref{xxlem2.8}(3).
\end{proof}

Ueyama \cite{U1} introduced the notion of a {\em graded isolated 
singularity}, and we recall his definition here. For a graded 
algebra $A$, let $\grmod A$ denote the category of finitely 
generated graded left $A$-modules. For a graded finitely 
generated $A$-module  an element $x \in M$ is called {\em torsion} 
if there exists a positive integer $n$ such that $ A_{\geq n} x= 0$. 
The module $M$ is called a {\em torsion module} if every 
element of $M$ is torsion. Let $\tors A$ denote the full 
subcategory of $\grmod A$ consisting of torsion modules.
We can then define the quotient 
category $\tails A = \grmod A / \tors A$. Following \cite{U1}, 
we say that $A^G$ has a {\em graded isolated singularity} 
if $\gldim(\tails A^G) < \infty$. Mori and Ueyama prove that 
if the Auslander map is an isomorphism, then $A^G$ has a 
graded isolated singularity if and only if $A\#G/I $ is 
finite-dimensional \cite[Theorem 3.10]{MU}. Examples of 
graded isolated singularities are of particular interest, 
since when $A^G$ has a graded isolated singularity, the 
category of graded CM $A^G$-modules has several nice 
properties (see \cite{U2}).

Next we compute the pertinency for $G$-coactions.   

\begin{lemma}
\label{xxlem2.11} 
Retain the hypothesis of Lemma {\rm{\ref{xxlem2.10}}}. 
\begin{enumerate}
\item[(1)]
If $g_2=1$ and $g_1\neq 1$ 
then $\Pty(A, \Bbbk^G)=3$. 
As a consequence, $A^{co\; G}$ has a graded isolated singularity.
\item[(2)]
If $g_2\neq 1$, $g_1\neq 1$, $g_1\neq g_2$, and $g_1^2=g_2^2=1$, 
then $\Pty(A, \Bbbk^G)\geq 2$.
\item[(3)]
If $g_1\neq 1$, $g_1^2=1$ and $g_2=g_1$, then $\Pty(A, \Bbbk^G)=3$.
As a consequence, $A^{co\; G}$ has a graded isolated singularity.
\item[(4)]
If $g_1^2\neq 1$, $g_2=g_1$, then $\Pty(A, \Bbbk^G)=3$.
As a consequence, $A^{co\; G}$ has a graded isolated singularity. 
\item[(5)]
If $g_1^2\neq 1$, and $g_2=g_1^{-1}$, then $\Pty(A, \Bbbk^G)\geq 2$.
\item[(6)]
If $G=T$, then $\Pty(A,\Bbbk^G)\geq 2$.
\end{enumerate}
\end{lemma}

\begin{proof} Let $J$ be the ideal generated by $\bigcap_{g\in G} AA_g$
as defined in \eqref{E2.1.3}.

(1) In this case $\deg_G x=g_1\neq 1$ and $\deg_G y=1$. Then 
$$\deg_{G} x^2= \deg_{G} xyx=\deg_{G} xy^2x=1.$$ 
It is easy to see that $x^2, xyx, xy^2 x\in J$. By the first 
relation of $A$, $y^3\in J$. Thus $A/J$ is finite dimensional, 
or $\GKdim A/J=0$. This means that $\Pty(A, \Bbbk^G)=3$, and 
by \cite[Corollary 3.8]{BHZ1}, $A^{co\; G}$ has a graded isolated 
singularity.

(2) It is easy to check that $xyx,yxy\in AA_{g}$ for all $g\in G$. 
So $xyx, yxy\in J$. Using relations of $A$, we have, in $A/J$,
$$\begin{aligned}
y^3&= a x^2 y+ b yx^2,\\
y^2 x&=xy^2+c x^3
\end{aligned}
$$
for some $a,b,c\in \Bbbk$. By using Bergman's Diamond Lemma 
\cite{Be} with degree lexicographic monomial order with $y >x$, 
$A/J$ has a monomial basis and each of the monomials
does not contain subwords $y^3$, $y^2x$, $xyx$, $yxy$. This 
implies that $A/J$ is spanned by
$$\{x^i: i\geq 0\}
\cup \{ y x^i, x^i y, x^i y^2: i\geq 0\} 
\cup \{ yx^i y, yx^i y^2: i\geq 0\}.$$
Thus $\GKdim A/J\leq 1$, and hence $\Pty(A,\Bbbk^G)\geq 2$.

(3,4) In these cases, every monomial of degree 4 is in $J$. So
$\GKdim A/J=0$ as required.

(5) In this case, one can show that $x^3\in J$ as $\deg_G x$ 
generates the group $G$. Similarly, we have $y^3\in J$.
Using the relations in $A$, one sees that, in $A/J$,
$$\begin{aligned}
x^3&=0,\\
y^3&=0,\\
yx^2&=-x^2y+a xyx,\\
y^2x&=-xy^2+a yxy
\end{aligned}
$$
for some $a\in \Bbbk$. By Bergman's Diamond Lemma \cite{Be},
$A/J$ is spanned by 
$$\{x^i (yx)^j y^k: 0\leq i,k\leq 2, j\geq 0\}.$$
Therefore $\GKdim A/J\leq 1$ and $\Pty(A,\Bbbk^G)\geq 2$.

(6) Let $\{d_7,\cdots,d_1\}$ be an ordered set of elements (possibly 
with repetitions) in $G$ such that the set 
$\{ \prod_{s=1}^7 d_s, \prod_{s=1}^6 d_s, \cdots, d_2 d_1, d_1\}$ 
is equal to $G\setminus \{1\}$. Suppose $f_s\in A$ are homogeneous
of degree $d_s$ for all $s=1,\cdots,7$. By Lemma \ref{xxlem1.9}(2), 
the product $f_7 f_6\cdots f_1$ is in $J$. Using this observation 
one sees that the following elements are in $J$:
$$y^2 x y^3 x, \; xyx^3yx, \; yx^3yx^2, \; x^3yx^3, \; 
x^2yx^3y, \; yxy^3xy, \; xy^3xy^2, \; y^3x y^3.$$
(The reason for verifying a product of 7 letters is that any 
subword of these monomials does not have 
$G$-degree 1. This list is  all degree 7 monomials in $J$.) 
Using the fact that 
$x=\frac{1}{2}(d+u)$ and $y=\frac{1}{2}(d-u)$, we obtain the 
following relation in $A/J$:
\begin{align}
\notag 
0&=2^7(x^3yx^3 - y^3xy^3)\\
\notag
&=(-2)u^7 + {\rm{lower}}\;\; {\rm{terms}}.
\end{align}
or equivalently,
\begin{equation}
\notag
u^7= {\rm{lower}}\;\; {\rm{terms}}
\end{equation}
In other words, we can write $u^7$ in terms of terms in lower degree
in the lexicographic order. 
Similarly, by using 
$x=\frac{1}{2}(d+u)$ and $y=\frac{1}{2}(d-u)$, we calculate
the following in $A/J$:
$$\begin{aligned}
0&=2^7(x^3yx^3 + y^3xy^3)\\
&= (-2a^5 -2a^4 +6a^3 +8a^2 -4a -6)udu^5 \\
&\quad + (-2a^6 +2a^5 +6a^4 -6a^3 -6a^2 +4a +2)udududu\\
&\quad + {\rm{lower}}\;\; {\rm{terms}};\\
0&=2^7(y^2xy^3x + x^2yx^3y)\\
&=(-2a^5 + 2a^4 + 10a^3 - 4a^2 - 12a + 2)udu^5 \\
&\quad + (-2a^6 - 2a^5 + 6a^4 + 6a^3 - 6a^2 -4a + 2)udududu\\
&\quad + {\rm{lower}}\;\; {\rm{terms}};\\
0&=2^7(xyx^3yx + yxy^3xy)\\
&= (2a^5 - 2a^4 - 6a^3 + 8a^2 + 4a - 6)udu^5\\
&\quad + (-2a^6 - 2a^5 + 6a^4 + 6a^3 - 6a^2 - 4a +2)udududu\\
&\quad + {\rm{lower}}\;\; {\rm{terms}};\\
0&=2^7(yx^3yx^2 + xy^3xy^2)\\
&= (2a^5 +2a^4 -10a^3 -4a^2 +12a+2)udu^5\\
&\quad +(-2a^6 +2a^5 +6a^4 -6a^3 -6a^2 +4a +2)udududu\\
&\quad + {\rm{lower}}\;\; {\rm{terms}};
\end{aligned}
$$
where ``lower term'' means a linear combination of 
monomials of degree 7 that have lower degrees than 
terms appearing the expression (in this case, 
$udududu$) with respect to lexicographic order.
Recall that $a$ is the scalar that appeared in one of the relations
of $A$, 
$$y^3=ax^2y+byx^2,$$
see the proof of part (2). 
If $a^2\neq 1$ and $a^2\neq 2$, then by a linear 
algebra computation, both $udu^5$ and $udududu$ 
can be expressed as  ``lower terms'':
\begin{align}
\notag
udu^5&= {\rm{lower}}\;\; {\rm{terms}},\\
\notag
udududu&={\rm{lower}}\;\; {\rm{terms}}.
\end{align}
Since $u^7$ and $udududu$ (and then $(ud)^4$) are equal 
to lower terms in $A/J$, Lemma \ref{xxlem1.10}(1) implies that 
$\GKdim A/J\leq 1$, as required.

If $a=1$, then we have
$$\begin{aligned}
0&=2^7(x^3yx^3 + y^3xy^3) \\
&= -6du^6 + 8d^3u^4 - 8d^4udu + 8d^5u^2 + 2d^7;\\
0&=2^7(y^2xy^3x + x^2yx^3y)\\
& = -4udu^5 + 2du^6 - 4d^2udu^3 + 4d^3u^4 - 4d^3udud 
     + 4d^4udu - 8d^5u^2 + 2d^7;\\
0&=2^7(xyx^3yx + yxy^3xy) \\
& = -2du^6 - 4d^3u^4 + 2d^7;\\
0&=2^7(yx^3yx^2 + xy^3xy^2)\\
& = 4udu^5 - 2du^6 + 4d^2udu^3 - 4d^3udud + 4d^4udu - 8d^5u^2 + 2d^7.
\end{aligned}
$$
By a linear algebra computation, we have
\begin{equation}
\notag
d^3 (ud)^2= {\rm{lower}}\;\; {\rm{terms}}.
\end{equation}
Since $u^7$ and $d^3 (ud)^2$ are equal to lower terms in 
$A/J$, Lemma \ref{xxlem1.10}(1) implies that 
$\GKdim A/J\leq 1$, as required.

If $a=-1$, then we have 
$$\begin{aligned}
0&= 2^7(x^3yx^3 + y^3xy^3) \\
&= 2du^6 - 4d^3u^4 + 2d^7;\\
0&= 2^7(y^2xy^3x + x^2yx^3y)\\
& = 4udu^5 + 2du^6 - 4d^2udu^3 - 4d^3udud + 4d^4udu + 8d^5u^2 + 2d^7;\\
0&= 2^7(xyx^3yx + yxy^3xy)\\
& = 6du^6 + 8d^3u^4 - 8d^4udu - 8d^5u^2 + 2d^7;\\
0&= 2^7(yx^3yx^2 + xy^3xy^2)\\
& = -4udu^5 - 2du^6 + 4d^2udu^3 + 4d^3u^4 - 4d^3udud + 4d^4udu + 8d^5u^2 + 2d^7.
\end{aligned}
$$
By a linear algebra computation, we have
\begin{equation}
\notag
d^3 (ud)^2= {\rm{lower}}\;\; {\rm{terms}}.
\end{equation}
Since $u^7$ and $d^3 (ud)^2$ are equal to lower terms in 
$A/J$, Lemma \ref{xxlem1.10}(1) implies that 
$\GKdim A/J\leq 1$, as required.

If $a^2=2$, we need to use a different set of elements in $J$.
By a similar argument as before and by Lemma \ref{xxlem1.9}(2),
the following elements are in $J$:

$$x^3 y^3 x^3, y^3x^3y^3,             
y^2 x y^3 x^3, x^2 y x^3 y^3, 
y^3 x^3 y x^2, x^3 y^3 x y^2,       
y x y^3 x y x^2, x y x^3 y x y^2,$$

$$y^2 x y^3 x y^2, x^2 y x^3 y x^2,   
y x y^3 x y^3, x y x^3 y x^3,       
y x^4 y x^3, x y^4 x y^3,           
x y^4 x y x^2, y x^4 y x y^2.$$      

By a Sage computation, we have, in $A/J$, 

$$
\begin{aligned}
0&= 2^9(x^3 y^3 x^3 - y^3x^3y^3)\\
&= (-2)u^9 + (8+4a)ududu^5 + (18+12a)ududududu+ {\rm{lower}}\;\; {\rm{terms}}\\
0&= 2^9(y^2 x y^3 x^3 - x^2 y x^3 y^3)\\
&=(-2)u^9 + (-4-4a)ududu^5 + (-2-4a)ududududu + {\rm{lower}}\;\; {\rm{terms}}\\
0&= 2^9(y^3 x^3 y x^2 - x^3 y^3 x y^2)\\
&= 2u^9 + (-4)ududu^5 + 2ududududu + {\rm{lower}}\;\; {\rm{terms}}\\
0&= 2^9(y^2 x y^3 x y^2 - x^2 y x^3 y x^2)\\
&= (-2)u^9 + 2ududududu + {\rm{lower}}\;\; {\rm{terms}}\\
\end{aligned}
$$

An easy linear algebra computation shows that
$$
\begin{aligned}
u^9 &= {\rm{lower}}\;\; {\rm{terms}}\\
(ud)^2u^5 &= {\rm{lower}}\;\; {\rm{terms}}\\
(ud)^4u &= {\rm{lower}}\;\; {\rm{terms}}\\
\end{aligned}
$$
Since $u^9$ and $(ud)^4u$ (and then $(ud)^5$) are equal 
to lower terms in $A/J$, Lemma \ref{xxlem1.10}(1) implies that 
$\GKdim A/J\leq 1$, as required. This finishes the proof.
\end{proof}

\subsection{Case 4: $A={\mathbb{F}}$, $x$ and $y$ are $G$-homogeneous}
\label{xxsec2.4} The final case is when $A={\mathbb{F}}$ and the proof is also quite 
tricky. We start with a result of \cite{CKZ}.

\begin{lemma}\cite[Lemma 1.6]{CKZ}
\label{xxlem2.12} Let ${\mathbb{F}}$ be generated by $x$ and $y$ and subject to two 
relations \eqref{E2.0.1}.
\begin{enumerate}
\item[(1)]
Define an order on monomials by extending $x<y$ lexicographically. 
Then we have a complete set of five relations that is the reduction system
in the sense of \cite[p.180]{Be}.
$$\begin{aligned}
y^3& =xyx,\\
yxy&=x^3,\\
y^2x^3&=xyx^2y,\\
yx^2yx&=x^3 y^2,\\
yx^4&=x^4y.
\end{aligned}
$$
\item[(2)] We also have the other relations:
$$\begin{aligned}
y^4& =x^4,\\
yxyx&=x^4,\\
xyxy&=y^4.
\end{aligned}
$$
\item[(3)]
There is a $\Bbbk$-linear basis consisting of the monomials
of the form
$$x^i (yx^3)^j (yx^2)^{\epsilon} (y^2x^2)^k y^a x^b$$
where $i,j,k\geq 0$, $\epsilon$ is either $0$ or $1$,
and
$$(a,b)=(0,0), (1,0), (1,1), (1,2), (1,3), (2,0), (2,1), (2,2)$$
if $j+\epsilon+k=0$,
$$(a,b)=(1,0), (1,1), (1,2), (1,3), (2,0), (2,1), (2,2)$$
if $j>0$ and $\epsilon+k=0$, and
$$(a,b)=(1,0), (2,0), (2,1), (2,2)$$
if $\epsilon+k>0$.
\end{enumerate}
\end{lemma}

Let $G$ be a finite group coacting on $A:={\mathbb{F}}$ such that $x$ and $y$ 
are $G$-homogeneous. Let $J$ be defined as in \eqref{E2.1.3}.

\begin{lemma}
\label{xxlem2.13}
Suppose there are $\alpha,\beta,\alpha',\beta'\geq 0$ such that
$$(yx^3)^{\alpha} (y^2x^2)^{\beta} \text{ and } 
(yx^3)^{\alpha'} (yx^2) (y^2x^2)^{\beta'}\in J,$$
then $\GKdim A/J\leq 1$. 
\end{lemma}

\begin{proof} We again use the Diamond Lemma \cite{Be}. 
By the fact that $x^4=y^4$ is central in $A$ 
[Lemma \ref{xxlem2.12}(2)], there is a monomial $f$ in $x$ and $y$ such 
that $f(yx^3)^i (y^2x^2)^k=x^{4w}$ for some $w$. So we have the 
following equations in $A/J$:
$$\begin{aligned}
x^{4w}&=0,\\
(yx^3)^{\alpha} (y^2x^2)^{\beta}&=0,\\
(yx^3)^{\alpha'} (yx^2) (y^2x^2)^{\beta'}&=0.
\end{aligned}
$$
Therefore, if there are $i, j, k$ with 
$x^i (yx^3)^j (yx^2)^{\epsilon} (y^2x^2)^k y^a x^b\neq 0$
in $A/J$ for $(a,b)$ as in Lemma \ref{xxlem2.12}(b), 
then there is a uniform bound on at least two of $\{i,j,k\}$. Then
$\GKdim A/J\leq 1$ by  \cite[(E1.1.6)]{BHZ2} and Lemma \ref{xxlem2.12}(3).
\end{proof}

The goal of the rest of this subsection is to find some monomials 
$(yx^3)^i (y^2x^2)^k$ and $(yx^3)^{i'} (yx^2) (y^2x^2)^{k'}$ in 
$J$. We introduce the following notation. Let
$$\begin{aligned}
X_0& =x,\\
X_1& =y.
\end{aligned}
$$
We will use $X_i$ for $i\in {\mathbb Z}/(2)$. Let
$$\begin{aligned}
V_0&= yx^3,\\
V_1&= xyx^2,\\
V_2&=x^2 yx,\\
V_3&= x^3 y,
\end{aligned}
$$
and 
$$\begin{aligned}
W_0&= y^2 x^2,\\
W_1&= xy^2 x,\\
W_2&= x^2 y^2,\\
W_3&= yx^2 y.
\end{aligned}
$$
We will use $V_i$ and $W_i$ for $i\in {\mathbb Z}/(4)$. The following lemma follows
from a direct computation and the relations given in Lemma \ref{xxlem2.12}.

\begin{lemma}
\label{xxlem2.14} Retain the above notation.
\begin{enumerate}
\item[(1)]
$X_i V_j= V_{j+1} X_i$ and $X_i W_j= W_{j+1} X_i$ 
for all $i\in {\mathbb Z}/(2)$ and $j\in {\mathbb Z}/(4)$.
\item[(2)]
Elements $\{V_0,\cdots, V_3, W_0,\cdots, W_3\}$ are pairwise commutative.
\item[(3)]
The following relations hold
\begin{enumerate}
\item[(a)]
$V_0 V_2=x^8=V_1 V_3$.
\item[(b)]
$W_0 W_2=x^8= W_1 W_3$.
\item[(c)]
$V_0 V_1=x^4 W_0$.
\end{enumerate}
\end{enumerate}
\end{lemma}

Now let $G$ be a finite group coacting on ${\mathbb{F}}$ such that $x$ and $y$ are 
$G$-homogeneous. We also assume that the $G$-coaction has trivial homological
codeterminant, namely, $\deg_G (x^4)=1$. Let

$$\begin{aligned}
x_i& =\deg_G X_i,\\
v_j& =\deg_G V_j,\\
w_k &=\deg_G W_k
\end{aligned}
$$
for $i\in {\mathbb Z}/(2)$ and $j,k\in {\mathbb Z}/(4)$. Let $N$ be the 
subgroup generated by $\{v_i\}_{i=0}^{3}\cup \{w_k\}_{k=0}^3$. By the  
lemma above, we have

\begin{lemma}
\label{xxlem2.15} Retain the above notation.
\begin{enumerate}
\item[(1)]
$G$ is generated by $x_0(=g_1)$ and $x_1(=g_2)$.
\item[(2)]
$N$ is an abelian subgroup of $G$.
\item[(3)]
$N$ is a normal subgroup of $G$.
\item[(4)]
$N$ is generated by $v_0$ and $v_1$ and 
$G/N$ is generated by the image $\overline{x_0}$ of $x_0$; 
and $\overline{x_0}=\overline{x_1}$ in $G/N$.
\item[(5)]
$N$ is also generated by $\{v_i, w_j\}$ for any pair $(i,j)$.
\item[(6)]
$G= N \cup N x_{i_1} 
\cup N x_{j_1}x_{j_2} \cup N x_{k_1}x_{k_2}x_{k_3}$ for any fixed $i_s, j_s,k_s
\in {\mathbb Z}/(2)$.
\item[(7)]
Let $n$ be the order of $v_0$. For any fixed $i_s\in {\mathbb Z}/(2)$ and 
$j_s, k_s\in {\mathbb Z}/(4)$ for $s=1,2,3,4$,
any element in $G$ is a right subword of 
$$v_{j_4}^n w_{k_4}^n x_{i_3} v_{j_3}^n w_{k_3}^nx_{i_2} 
v_{j_2}^n w_{k_2}^nx_{i_1} v_{j_1}^n w_{k_1}^n.$$
\end{enumerate}
\end{lemma}

\begin{proof} (1) Since $G$-coaction is inner-faithful, $G$ is generated by
$\deg_G x$ and $\deg_G y$.

(2) This follows from Lemma \ref{xxlem2.14}(2).

(3) This follows from Lemma \ref{xxlem2.14}(1) and part (1).

(4) Since $\deg_G x^4=1$, $x_0^4=x_1^4=1$. By Lemma \ref{xxlem2.14}(3),
$v_2=v_0^{-1}$, $v_3=v_1^{-1}$, and $w_0=v_0v_1$. It is easy to check
from Lemma \ref{xxlem2.14}(1) that $w_1=v_0^{-1}v_1$, $w_2=v_0^{-1}v_1^{-1}$
and $w_3=v_0 v_1^{-1}$. Therefore $N$ is generated by $v_0$ and $v_1$. 
It is clear that in $G/N$, $\overline{x_0}=\overline{x_1}$. So $G/N$ is 
generated by $\overline{x_0}$.

(5) By the proof of part (4), we have 
$$\begin{aligned}
\{v_0,v_1,v_2,v_3\}&=\{v_0, v_1, v_0^{-1},v_1^{-1}\}, \\
\{w_0, w_1, w_2, w_3\}&=\{ v_0v_1, v_0^{-1} v_1, v_0^{-1}v_1^{-1}, v_0 v_1^{-1}\}.
\end{aligned}
$$
Therefore $N$ is generated by $\{v_i, w_j\}$  for any pair $(i,j)$.

(6) This follows from the fact that $\overline{x_0}=\overline{x_1}$ in the 
quotient group $G/N$ and that $G/N\cong {\mathbb Z}/(4)$.

(7) This follows from parts (5,6).
\end{proof}

We have a version of Lemma \ref{xxlem2.15} for monomials in ${\mathbb{F}}$. 

\begin{lemma}
\label{xxlem2.16}
Retain the above notation.
\begin{enumerate}
\item[(1)]
For any fixed $i_s\in {\mathbb Z}/(2)$ and 
$j_s, k_s\in {\mathbb Z}/(4)$ for $s=1,2,3,4$,
any element in $G$ is the degree of some right subword of 
$$\Phi:=V_{j_4}^n W_{k_4}^n X_{i_3} V_{j_3}^n W_{k_3}^nX_{i_2} 
V_{j_2}^n W_{k_2}^nX_{i_1} V_{j_1}^n W_{k_1}^n.$$
As a consequence, $\Phi\in J$. 
\item[(2)]
$(yx^3)^{4n} (yx^2) (y^2 x^2)^{4n}$
and $(yx^3)^{4n} (y^2 x^2)^{4n+1}$ are elements in $J$.
\end{enumerate}
\end{lemma}

\begin{proof} (1) This is a slightly stronger version of Lemma \ref{xxlem2.15}(7).
For example, if an element $g$ is of the form 
$$v_{j_2}^a w_{k_2}^b x_{0} v_{j_1}^n w_{k_1}^n$$
then we take a right subword of the form $f:=V_{j_2}^a W_{k_2}^b X_{0} V_{j_1}^n W_{k_1}^n$.
Since $V_s$ and $W_t$ all commute, $f$ is a subword of $\Phi$.
Clearly, $\deg_G f=v_{j_2}^a w_{k_2}^b x_{0} v_{j_1}^n w_{k_1}^n.$ The consequence
follows from Lemma \ref{xxlem1.9}(1). 

(2) By using Lemma \ref{xxlem2.14}(1), $\Phi$ equals
$$V_{j_4}^n V_{j_3+1}^n V_{j_2+2}^n V_{j_1+3}^n X_{i_3}X_{i_2}X_{i_1} 
W_{k_4-3}W_{k_3-2}W_{k_2-1}W_{k_1}.$$
By taking $j_4=j_3+1=j_2+2=j_1+3=0$ and $k_4-3=k_3-2=k_2-1=k_1=0$
and $i_3=1$, $i_2=i_1=0$, we have that 
$$ (yx^3)^{4n} (yx^2) (y^2 x^2)^{4n}\in J.$$

By taking $j_4=j_3+1=j_2+2=j_1+3=0$ and $k_4-3=k_3-2=k_2-1=k_1=1$
and $i_3=i_2=1$ and $i_1=0$, we have that 
$$ (yx^3)^{4n} (y^2x) (xy^2 x)^{4n}\in J.$$
Then $(yx^3)^{4n} (y^2 x^2)^{4n+1}\in J$.
\end{proof}

Now we can prove the result of this subsection.

\begin{proposition}
\label{xxpro2.17}
Retain the notation as in Theorem {\rm{\ref{xxthm0.1}}}. Suppose
that $A={\mathbb{F}}$ and $x$ and $y$ are $G$-homogeneous. Then 
$\Pty(A,\Bbbk^G)\geq 2$.
\end{proposition}

\begin{proof} Combining Lemma \ref{xxlem2.16}(2) with 
Lemma \ref{xxlem2.13}, $\GKdim A/J\leq 1$. This is equivalent
to $\Pty(A,\Bbbk^G)\geq 2$.
\end{proof}

Putting all these pieces together we have a proof of Theorem \ref{xxthm0.1}.

\begin{proof}[Proof of Theorem \ref{xxthm0.1}]
First we assume that $G$ is abelian (that could be cyclic).
If $A$ is not $\mathbb{D}(\alpha,-1)$ for some $\alpha\neq 2$, 
the assertion follows from \cite[Theorem 0.6]{BHZ2}.
Now we assume that $A=\mathbb{D}(\alpha,-1)$ for some $\alpha\neq 2$.
Using notations introduced in subsection \ref{xxsec2.3}, the character 
group of $G$ is denoted by $\widehat{G}$. If $\widehat{G}$
acts on $A$ diagonally with respect to the basis $\{d,u\}$, then 
the assertion follows from \cite[Proof of Theorem 0.6]{BHZ2}.
Otherwise, $\widehat{G}$ contains non-diagonal matrices with 
respect to the basis $\{d,u\}$. In this case the assertion 
follows from Lemmas \ref{xxlem2.9}, \ref{xxlem2.10} and 
\ref{xxlem2.11}. This takes care of the case when $G$ is abelian.

Next we assume that $G$ is not abelian. As a consequence, 
$G$ is not cyclic, so we can apply Lemma \ref{xxlem2.1}.
Using the classification given in Lemma \ref{xxlem2.1}, we 
only need to show the assertion for the first three cases 
as in the last two cases $G$ is abelian.

In case Lemma \ref{xxlem2.1}(1), this is Proposition \ref{xxpro2.4}.

In case Lemma \ref{xxlem2.1}(2), this is Proposition \ref{xxpro2.17}.

In case Lemma \ref{xxlem2.1}(3), this is Lemma \ref{xxlem2.7}(3). 

This finishes the proof.
\end{proof}

\section{Examples} 
\label{xxsec3}

We conclude this paper with several examples that indicate the 
variety of covariant subrings that can be obtained from 
coactions on down-up algebras, and give some concluding 
comments.

\begin{example}
\label{xxex3.1} 
A group $G = \langle a, b \rangle$ coacts on $\mathbb{D}_\beta$, 
where $\deg_G d = a$ and $\deg_G u = b$, with trivial 
homological codeterminant if $a^2b^2 =1$ [Lemma \ref{xxlem1.5}].  
One such family of groups is the dihedral groups, for $n\geq 2$,
$$D_{2n} = \langle a,b \;| \;a^2 = b^2 = (ba)^n =1 \rangle.$$ 
Each element of $\mathbb{D}_\beta$ can be written as a linear 
combination of monomials of the form $d^i(ud)^ju^k$, and such 
a monomial is in the identity component of $\mathbb{D}_\beta$ exactly when
$$a^i(ba)^jb^k = 1.$$
Clearly $d^2$ and $u^2$ are covariants, so it suffices to consider 
the four cases $(i,k) = (0,0), (1,0),(0,1), (1,1)$, and it is 
easy to check that the subring of covariants is generated as a 
$\Bbbk$-algebra by $d^2, u^2, (du)^n, (ud)^n$, i.e.  
$$\mathbb{D}_\beta^{co\; D_{2n}} = \Bbbk \langle  d^2, u^2, (du)^n, 
(ud)^n\rangle.$$  
When $\beta = \pm 1$,  $\mathbb{D}_{\pm 1}^{co\; D_{2n}}$ is isomorphic 
to the commutative algebra 
$$\mathbb{D}_{\pm 1}^{co\; D_{2n}} \cong \frac{\Bbbk[X,Y,Z,W]}{(X^nY^n-ZW)},$$
a hypersurface in  ${\mathbb A}^4$. When $\beta \neq \pm 1$, the 
ring of covariants is a hypersurface in a noncommutative skew 
polynomial ring in the sense of \cite[Definition 1.3(c)]{KKZ3}. 
For any $\beta\neq 0$, $\mathbb{D}_\beta \# \Bbbk^{D_{2n}}$ is a 
noncommutative quasi-resolution or NQR (a generalization of NCCR) 
of $\mathbb{D}_\beta^{co\; D_{2n}}$ in the sense of \cite{QWZ}.  When 
$n=4$, this example should be compared to \cite[Example 2.1]{CKZ}, 
where a different coaction of $D_8$ (without trivial homological 
codeterminant) on $\mathbb{D}_1$ is given; in that example the ring of 
covariants is a commutative hypersurface in ${\mathbb A}^4$, but 
Auslander's Theorem fails [Remark \ref{xxrem1.6}(1)]. 
\end{example}

Next we consider a second coaction on $\mathbb{D}_\beta $, 
where the ring of covariants is quite different.

\begin{example}
\label{xxex3.2}
The quaternion group $G=Q_8$ of order 8 coacts on $\mathbb{D}_\beta$ by
$\deg_G d = i$ and $\deg_G u = k$ with trivial homological 
determinant [Lemma \ref{xxlem1.5}].  A monomial 
$d^{e_1}(ud)^{e_2}u^{e_3}$ has group grade the identity of $Q_8$ 
exactly when
$$i^{e_1}j^{e_2}k^{e_3} = 1$$
holds in $Q_8$. It is not hard to check that the covariants are 
generated by the following 9 monomials;
$$d^4, \; u^4, \; d^2u^2, \; d^2(ud)^2,\; (ud)^2u^2,$$
$$ d(ud)u^3 = (du)^2u^2,\; d^3(ud)u = d^2(du)^2, \; 
d(ud)^3 u =(du)^4, \; (ud)^4.$$
When $\beta \neq 0$, $\mathbb{D}_{\beta} \# \Bbbk^{Q_{8}}$ is a 
noncommutative quasi-resolution (NQR) of the covariant subring 
$\mathbb{D}_{\beta}^{co\; Q_{8}}$ in the sense of \cite{QWZ}.  
\end{example}

Finally we consider the down-up algebra ${\mathbb{H}}$.

\begin{example}
\label{xxex3.3}
The dihedral groups $D_{2n}$ coact on ${\mathbb{H}}$ homogeneously with 
trivial homological codeterminant, although our proof of Auslander's 
theorem holds for any group coaction in this case (see Lemma \ref{xxlem2.7}
and the comments after that).  Suppose that $G:=D_{2n}=\langle a,b\mid 
a^2=b^2=1=(ab)^n \rangle$ and that $\deg_G x = a$ and $\deg_G y = b$.
The relations in ${\mathbb{H}}$ \eqref{E2.0.2} can be written as
$$(x^2-y^2)y = -y (x^2 - y^2)$$
$$x(x^2-y^2) = -(x^2 - y^2)x,$$
and hence $x^2-y^2$ is a normal element of ${\mathbb{H}}$, and, moreover, 
$xy$ and $yx$ commute with $y^2-x^2$. It is clear that $x^2$ and $y^2$ 
are  covariants under this action, and that
${\mathbb{H}}/(x^2-y^2) \cong \Bbbk \langle x, y \rangle /(x^2 - y^2)$. Since
$x^2-y^2$ is also a normal element of ${\mathbb{H}}^{co \; D_{2n}}$, we obtain that
$$\frac{{\mathbb{H}}^{co \; D_{2n}}}{(x^2-y^2)} \cong 
\left(\frac{\Bbbk \langle x, y \rangle}{(x^2-y^2)}\right)^{co\; D_{2n}}.$$
It follows that the generators of ${\mathbb{H}}^{co \; D_{2n}}$ are the 4 elements
$$x^2, \; y^2, \; (yx)^n, \; x(yx)^{n-1} y = (xy)^n.$$

Next we show that ${\mathbb{H}}^{co \; D_{2n}}$ is a hypersurface 
\cite[Definition 1.3(c)]{KKZ3} in an iterated Ore extension that is 
an AS regular algebra of dimension 4.

Multiplying the relation $x^2 y+yx^2 -2y^3=0$ by $y$ on each side 
and subtracting gives the relation $x^2y^2 - y^2x^2=0$.  We 
next give a number of relations in ${\mathbb{H}}$; note that the 
defining relations of ${\mathbb{H}}$ are symmetric in $x$ and $y$, 
so the relations with $x$ and $y$ interchanged also hold. 
It is  easy to check the following relation 
$$x^2y - yx^2 = 2 y (y^2-x^2);$$
multiplying by $x$ on left gives:
$$[x^2,(xy)] = 2xy(y^2-x^2) = (y^2-x^2)(2xy).$$
Similarly 
$$[y^2,(xy)] = 2xy(y^2-x^2).$$
Inductively one can show that $$[y^2,(xy)^n] = 2n (xy)^n(y^2-x^2).$$
Further
$$(yx)(xy) =y^2(y^2+(y^2-x^2)),$$
and inductively we get
$$\begin{aligned}
& (yx)^n(xy)^n = \\
&y^2(y^2+(y^2-x^2))(y^2+2(y^2-x^2))(y^2+3(y^2-x^2)) 
\cdots (y^2+(2n-1)(y^2-x^2)).
\end{aligned}
$$
Let the right side of the above equation be denoted by $w(x^2,y^2)$.
We claim that the subalgebra generated by $x^2, \; y^2, \; (yx)^n, 
\; (xy)^n$ is a generalized Weyl algebra (or ambiskew polynomial ring, 
as in \cite{Jo}), i.e. it is an iterated Ore algebra modulo one relation.

To simplify notation, let $X=x^2, \; Y= y^2, \; Z^+ = (yx)^n, \; 
Z^{-} = (xy)^n$.  From the relations above we have the following relations:
$$\begin{aligned}
XY&=YX,\\
Z^{+} X & = (X+ 2n(Y-X))Z^+,\\
Z^{+} Y & = (Y+2n(Y-X))Z^+.
\end{aligned}
$$
In this notation $Z^+ Z^- = w(X,Y)$ and $Z^- Z^+ = w(Y,X)$.
Let $B$ be the algebra generated by $X,Y,Z^+$ defined by the 
first three relations above. Then $B$ is the Ore extension 
$\Bbbk[X,Y][Z^+; \sigma]$, where $\sigma$ is the automorphism 
of $\Bbbk[X,Y]$ given by $\sigma(X) = X + 2n(Y-X)$ and 
$\sigma(Y) = Y + 2n(Y-X)$.  Adjoining $Z^-$ to $B$ adds the 
following three relations:
$$\begin{aligned}
Z^{-} X &= (X- 2n(Y-X))Z^-,\\
Z^- Y   &= (Y-2n(Y-X))Z^-,\\
Z^-Z^+  &= Z^+Z^-  -  f(X,Y).
\end{aligned}
$$
where $f(X,Y) = w(X,Y)-w(Y,X) \in \Bbbk[X,Y]$.
One checks that $\sigma^{-1}(X) = X - 2n(Y-X)$,
$\sigma^{-1}(Y) = Y - 2n(Y-X)$ and $\sigma^{-1}(Z^+)=Z^+$ defines
an algebra automorphism of $B$.  Define the map $\delta$ 
on $B$, by $\delta(X) = \delta(Y) = 0$ and $\delta(Z^+) = -f(X,Y)$.  
One checks that $\delta$ extends to a $\sigma^{-1}$-skew 
derivation of $B$, preserving the three relations 
defining $B$.  Hence we have
$${\mathbb{H}}^{co\; D_{2n}} \cong 
\frac{\Bbbk[X,Y][Z^+;\sigma][Z^-; \sigma^{-1}, \delta]}{(Z^+Z^- -w(X,Y))},$$
and ${\mathbb{H}}^{co\; D_{2n}}$ is a hypersurface \cite[Definition 1.3(c)]{KKZ3}
in an AS regular algebra of dimension 4, with Hilbert series
$$\frac{1 - t^{4n}}{(1-t^2)^2(1-t^{2n})^2}.$$
\end{example}

Down-up algebras have no reflections \cite[Proposition 6.4]{KKZ4}, 
so we would expect Auslander's theorem to hold for all finite group 
actions;  this has been proved for almost all finite group actions, 
except for some finite groups acting on ${\mathbb D}(\alpha,-1)$ 
for $\alpha \neq 2$ (abelian groups with trivial homological 
determinant are covered by Theorem \ref{xxthm0.1}).  Theorem 
\ref{xxthm0.1} also covers finite group coactions with trivial 
homological determinant, and we have shown Theorem \ref{xxthm0.1} 
does not hold for all group coactions [Remark \ref{xxrem1.6}(1)]. 
It remains to examine actions by other Hopf algebras on down-up 
algebras.

\subsection*{Acknowledgments}
The authors thank Kenneth Chan and Zhibin Gao for several conversations 
on the subject and their help with some computations in this paper and
thank the referees for their very careful reading and extremely valuable 
comments. J. Chen was partially supported by the National Natural Science 
Foundation of China (Grant No. 11571286) and the Natural Science
Foundation of Fujian Province of China (Grant No. 2016J01031).
E. Kirkman was partially supported by grant \#208314 from the Simons 
Foundation. J.J. Zhang was partially supported by the US National 
Science Foundation (Grant Nos. DMS-1402863 and DMS-1700825).

\end{document}